\documentclass [11pt,reqno]{amsart}
\usepackage {amsmath,amssymb,verbatim,geometry}
\usepackage[all]{xy}
\newif\ifpdf
\ifpdf
\usepackage[pdftex]{hyperref}
\else
\fi

\numberwithin{equation}{section}       

\newtheorem{prop} {Proposition} [section]
\newtheorem{thm}[prop] {Theorem}

\newtheorem{lem}[prop] {Lemma}
\newtheorem{cor}[prop]{Corollary}
\newtheorem{prop-def}[prop]{Proposition-Definition}

\newtheorem{exa}[prop]{Example}

\theoremstyle{remark}
\newtheorem{rem}[prop]{Remark}
\newtheorem*{ackn}{Acknowledgements}

\newcommand{\C}{{\mathbb{C}}}

\newcommand{\R}{{\mathbb{R}}}

\newcommand{\D}{\Delta}

\title[Regularity of geodesics]{Regularity of geodesics in the spaces of convex and plurisubharmonic functions}

\date{\today}

\author{Soufian Abja, S\l awomir Dinew}

 \address{Institute of Mathematics, Jagiellonian University, ul Lojasiewicza 6, 30-348 Krak\'ow, Poland}
\email{Soufian.Abja@im.uj.edu.pl}
 \address{Institute of Mathematics, Jagiellonian University, ul Lojasiewicza 6, 30-348 Krak\'ow, Poland}
\email{Slawomir.Dinew@im.uj.edu.pl }

\begin{document}
\begin{abstract}
 In this note we investigate the regularity of geodesics in the space of convex and plurisubharmonic functions. In the real setting
 we prove (optimal) local $C^{1,1}$ regularity. We construct examples which prove that the global $C^{1,1}$ regularity  fails both in the real and complex case
 in contrast to the K\"ahler manifold setting. Finally we show a necessary and sufficient conditions for existence of a smooth geodesic between
 two smooth strictly convex functions. 
\end{abstract}

\maketitle

\section*{Introduction} Given a compact K\"ahler manifold $(X,\omega)$ the space of {\it K\"ahler potentials} is defined by
$$\mathcal P(X,\omega):=\lbrace u\in C^{\infty}(X,\mathbb R)|\ \omega+i\partial\overline{\partial} u>0\rbrace.$$

A classical construction of Mabuchi \cite{Mab87} endows the space 
$\mathcal P(X,\omega)$ with the structure of an infinite-dimensional Riemannian manifold. 
More precisely at each $u\in \mathcal P(X,\omega)$ the tangent space $T_u\mathcal P(X,\omega)$ is naturally identified with $C^{\infty}(X,\mathbb R)$ and the scalar product between two
{\it vectors} $f,g\in T_u\mathcal P(X,\omega)$ is given by
\begin{equation}\label{L2kahlercase}
 \langle f,g\rangle_u:=\int_Xfg(\omega+i\partial\overline{\partial} u)^n,
\end{equation}
where $n:=dim_{\mathbb C}X$. 

This abstract construction has attracted a lot of interest after the works of Semmes \cite{Sem92} and Donaldson\cite{D99}. In these papers it was shown that for a curve $u_t$ in $\mathcal P(X,\omega)$
the geodesic equation 
$$\ddot{u}_t-|\nabla \dot{u}|^2_{\omega+i\partial\bar{\partial}u}=0$$
in the above setting can be rewritten as a homogeneous complex Monge-Amp\`ere equation. Since then the notion of geodesics in the space of K\"{a}hler metrics
on compact K\"ahler manifolds has been playing a prominent role in K\"{a}hler geometry and has found a lot of applications especially in the uniqueness problem for extremal K\"ahler
metrics- see \cite{BB17} and references therein. 

A major analytical issue in the study of geodesics is their {\it optimal regularity}. The result of Chen \cite{Che00}, with complements by Blocki \cite{Bl12}, shows that geodesics  
always have bounded space-time Laplacian (so in particular they are $C^{1,\alpha}$-smooth
for any $\alpha< 1$). Recently, Chu-Tosatti-Weinkove \cite{CTW17} proved that  the geodesics are globally  $C^{1,1}$-regular in space and time directions.

In a similar vein such a metric construction has been applied to the space of {\it plurisubharmonic functions} in strictly pseudoconvex domains-\cite{Ras17, Abj19}. In this setting
we consider $\Omega\Subset\mathbb{C}^n$- a smoothly  bounded, strictly pseudoconvex domain: in particular
there exists a smooth function  $\rho$ defined  in neighborhood  $\Omega '$ of $\bar{\Omega}$
such that  
$$\Omega=\{z\in \Omega' \;\;| \rho(z)<0\},\ \partial\Omega=\{\rho=0\}$$
 with $d\rho\neq 0$ on $\partial\Omega$ and  $dd^c \rho >0$ on $\overline{\Omega}$,
where
$$
d:= \partial+ \bar{\partial}\;,\; d^c:=\frac{i}{2\pi}(\partial-\bar{\partial}).
$$

Then the  Mabuchi space $\mathcal {H}$ of smooth, strictly plurisubharmonic functions is defined by
 $$\mathcal{H}:=\{\varphi\in C^{\infty}(\bar{\Omega})|\ dd^c\varphi>0\ {\rm in}\ \overline{\Omega}, \varphi=0\; {\rm on}\ \partial\Omega\}.$$
   The space $\mathcal {H}$ again can be endowed with the structure of an infinite dimensional Riemannian manifold, whose tangent space 
   $T_{\varphi}\mathcal{H}$ can be identified with the set of functions
   in  $C^{\infty}(\bar{\Omega},\R)$, vanishing  at the boundary of $\Omega$. The $L^2$ Mabuchi metric on $\mathcal{H}$ is given by  
   $$ 
\langle\psi_1,\psi_2\rangle_\varphi:=\int_{\Omega}\psi_1\psi_2 (dd^c\varphi)^n,
$$
for any $\varphi\in\mathcal{H},\psi_1,\psi_2 \in T_{\varphi}\mathcal{H}$.
 The geodesics between two points $\varphi_0$, $\varphi_1$ in $\mathcal{H}$ are defined as the minimizers of
the energy functional
$$
\varphi\longmapsto H(\varphi):= \frac{1}{2}\int_{0}^1\int_{\Omega}(\dot{\varphi}_t)^2( dd^c\varphi_t)^n
,$$
where $\varphi=\varphi_t$ is a path in $\mathcal{H}$ joining $\varphi_0$ and $\varphi_1$. The geodesic equation is obtained by computing the Euler-Lagrange equation of the functional $H$.
It reads
\begin{equation*}\label{geodesicPSH}
    \ddot{\varphi}(t)-|\nabla\;\dot{\varphi}(t)|^{2}_{dd^c\varphi(t)}=0,\;\;
\end{equation*}
where $\nabla$ is the gradient with respect to the metric $dd^c\varphi$.

Just as in the K\"ahler case the existence of a geodesic between  any two points $\varphi_0$ and $\varphi_1$ from $\mathcal{H}$ reduces to the solution of the following Dirichlet problem
\begin{equation}\label{MAcomplex}
\begin{cases}
  \phi\in PSH(\Omega\times A)\cap C(\overline{\Omega\times A});\\
    (dd^c_{z,\zeta}\phi)^{n+1}=0 &{\rm in}\  \hbox{$\Omega\times A$}; \\
    \phi=\varphi_0 &{\rm in}\ \hbox{$\Omega\times \{|z|=1\}$}; \\
   \phi=\varphi_1 &{\rm in}\ \hbox{$\Omega\times \{|z|=e\}$}; \\
    \phi=0 &{\rm in}\ \hbox{$\partial \Omega\times A$,}
 \end{cases} 
\end{equation}
where $PSH$ stands for the class of plurisubharmonic functions,  $A=\{z\in \C| 1<|z|<e\}$ denotes an annulus in $\C$  and $\phi(z,\zeta)=\varphi_t(z)$ with $t=\log|\zeta|$. More precisely if the solution
$\phi$ is $C^{\infty}$ smooth and strictly plurisubharmonic in the space variables then $t\rightarrow\phi(z,t)$ is the geodesic joining $\varphi_0$ and $\varphi_1$.

The above equation is known as the homogeneous complex Monge-Amp\`ere equation. 

Looking at the {\it real} counterpart of the constructions above it is natural to consider smoothly bounded {\it strictly convex} domains
$U$. Then the analog of $\mathcal H$ is the space $\mathcal S$ of {\it strictly convex} functions i.e.

$$\mathcal{S}:=\{u\in C^{\infty}(\overline{U})| D^2u>0\ {\rm on}\ \overline{U} , u=0\; {\rm on}\ \partial\Omega\}.$$

The $L^2$ metric is given by

$$\langle f_1,f_2\rangle_u:=\int_{U}f_1f_2det(D^2u)$$
for any $u\in\mathcal{S}\ ,f_1,\ f_2 \in T_{u}\mathcal{S}$, where the tangent space is again identified with the 
functions in $C^{\infty}(\bar{U},\mathbb R)$ that vanish at $\partial U$.

Not surprisingly the corresponding geodesic equation can also be rewritten as a homogeneous {\it real} Monge-Amp\`ere equation. Similarly to the complex case we seek a convex function $u$ in $U\times (0,1)$, continuous up to the boundary of $U\times(0,1)$,  such that
\begin{equation}\label{ss}
\begin{cases}
\det(D^2 _{x,t}u)=0\;{\rm in}\ \;\; U\times (0,1);\\
u=\varphi_0\;\; {\rm in}\ \;\; U\times \{0\};\\
u=\varphi_1\;\; {\rm in}\ \;\; U\times \{1\};\\
u=0\;\; {\rm in}\ \;\; \partial U\times (0,1),
\end{cases}
\end{equation}
where $D^2_{x,t}u$ is the Hessian of $u$ with respect to $(x,t)$ in $\Omega\times (0,1)$, and $\varphi_0$, $\varphi_1$ are two strictly convex $C^{\infty}$-smooth functions vanishing on $\partial{U}$.
Again $u(x,t)$ will be the geodesic provided that the solution is smooth and strictly convex in the space variables.

It is a basic fact that suitably defined {\it weak} solutions to (\ref{MAcomplex}) and (\ref{ss}) exist and are unique- see the next section for more details. It is customary
to call these solutions {\it weak geodesics} although, strictly speaking, these need not be curves in $\mathcal H$ or $\mathcal S$, respectively.

Just as in the K\"ahler case the optimal regularity of weak geodesics is one of the main problems in the theory. A lot is known about the regularity of the solutions in smoothly bounded strictly convex 
domains- see \cite{CNS86} or on general smooth domains in presence of subsolutions- \cite{Gu}. In the case of strip type unbounded domains the regularity issues were analyzed in \cite{LW15}. In particular the optimal regularity one can expect in general is $C^{1,1}$- see \cite{CNS86}.

Back to problem (\ref{ss}) we remark that the boundary
data  is of very special form. In our setting, 
however, the domain $U\times(0,1)$ has corners and the
regularity of the solutions in such a domain is not sufficiently well understood.  In the complex setting even less is known- see \cite{Abj19} for some partial results.

The goal of the note is to prove global Lipschitz and almost global $C^{1,1}$ estimates for the problem (\ref{ss}):

\begin{thm}\label{22} Let $U$ be strictly convex domain, $\varphi_0$, $\varphi_1$ be two functions from $\mathcal S$. Then the weak geodesic $u$ is $C^{1,1}$ regular away from corner points 
$\partial U\times\lbrace0\rbrace\cup \partial U\times\lbrace1\rbrace$. For any neighborhood $W$ of $\partial U\times\lbrace0,1\rbrace$ such that $U\times[0,1]\setminus W$ is
convex
there is a constant $C_W$ dependent on $W, U$ and
$\varphi_0,\varphi_1$ such that
$$||u||_{C^{1,1}(\overline{U\times[0,1]}\setminus W)}\leq C_W.$$
\end{thm}

We remark that the above result is stronger than merely local regularity, as it shows that potential blow-up may occur only at the corners.

Analogous regularity for weak plurisubharmonic geodesics is not known (see \cite{Abj19} for partial results in the case of $\Omega$ being an Euclidean ball).
We can however show the corresponding result for weak geodesics in $\mathcal H$ joining two {\it toric} plurisubharmonic functions. Recall that a domain $\Omega\subset \mathbb C^n$ is {\it Reinhardt} if it is 
invariant with respect to the standard $n$-dimensional torus action on the coordinates. The {\it axis set} is simply  
$$M:=\lbrace z\in \Omega\ |\ \exists i\in\lbrace1,\cdots,n\rbrace:\ z_i=0\rbrace.$$
\begin{cor}\label{corol}
 Let $\Omega$ be a smoothly bounded strictly pseudoconvex Reinhardt domain. Suppose that $\phi$ is a weak geodesic solving the problem (\ref{MAcomplex}). If $\varphi_1,\varphi_2\in\mathcal H$ are toric in the space variables i.e.
 for all $z,z'\in \Omega$ satisfying $|z|=|z'|$ one has $\varphi_j(z)=\varphi_j(z'),\ j=1,2$ then $\phi$ is $C^{1,1}$ away from the corner $\partial\Omega\times\partial A$ 
 and  $M\times A$.
\end{cor}

Given these results and the theory in the K\"ahler case it is natural to ask whether global $C^{1,1}$ bounds could be obtained. A bit surprisingly we show (see Examples \ref{realexample} and \ref{complexexample})
that this is {\bf not} the case: there exist pairs of points in $\mathcal S$ and $\mathcal H$ such that the weak geodesics joining them
are not globally $C^{1,1}$.

With Example \ref{realexample} in mind it is natural to ask what are the exact conditions guaranteeing that two points $\varphi,\psi\in\mathcal S$ can be joined by a smooth geodesic (i.e. problem (\ref{ss}) admits a
solution $u$ which is smooth in space and time
and is furthermore strictly convex up to the boundary for a fixed time). Exploiting the ideas of Li and Wang from \cite{LW15} we can
get an exact answer- our second main result in this note:
\begin{thm}\label{smoothgeodesic}
Let $\varphi,\psi\in\mathcal S$. Then $\varphi$ and $\psi$ can be joined by smooth geodesic if and only if the gradient image of $\varphi$
$$\partial \varphi(U):=\lbrace p\in\mathbb R^n|\ \exists x\in U:\ Du(x)=p\rbrace$$
equals the gradient image of $\psi$.  
\end{thm}
We remark that Theorem \ref{smoothgeodesic} shares some similarities with Guan's theorem on existence of smooth geodesics in the case
of {\it toric} compact K\"ahler manifolds- see \cite{G}. Indeed, smooth geodesics always exist in the toric setting, but also the image of the {\it moment maps} (the analogue of the gradient image) is fixed- it is equal to the Delzant polytope of the toric manifold.

We also remak that  the problems of finding criteria for existence of smooth geodesics both in the general K\"ahler and plurisubharmonic setting are widely open.

This paper is organized as follows. In Section 1 we recall some preliminary results. We prove Theorem \ref{22} and Corollary \ref{corol}  in the next section. In Section 3 
we present an example in which the regularity at the corner points fails to be $C^{1,1}$. In Section 4 a complex analogue of such an example is constructed. In the last section we prove Theorem \ref{smoothgeodesic}.
\begin{ackn} Both Authors were supported by Polish National Science Centre grant 2017/26/E/ST1/00955.
\end{ackn}
\section{Preliminaries}
In this section we gather the notions and results that will be used later on.

For the basics of the theory of weak solutions of the real Monge-Amp\`ere equation we refer to \cite{Gut} or \cite{Fi}. The complex counterpart can be found in \cite{Kol05}.

 In the study of homogeneous Monge-Amp\`ere Dirichlet problem the natural solution is the envelope. The envelope associated to (\ref{ss}) is defined as follows.
 $$u(x,t)=\sup\{v(x,t) \in C(\bar{U}\times [0,1]),\; v-{\rm convex},  v\leq \phi\; on \;\partial(U\times (0,1) )\},$$
 where $\phi|_{\partial U\times (0,1)}=0$, $\phi|_{U\times\{0\}}=\varphi_0$ and $\phi|_{U\times\{1\}}=\varphi_1.$
 
It is a classical fact that the envelope $u$ solves $det(D^2u)=0$ in a weak sense (see \cite{Gut} for details). It also matches the given boundary values provided that there is
a convex function matching this data. In our case the function
$$h(x,t)=\max(\varphi_0(x)-Ct, \varphi_1(x)-C(1-t), \varphi_0(x)+\varphi_1(x))$$ 
is convex with respect to  $(x,t)$ and assumes the given boundary values for sufficiently large constant $C>0$ (this function is modelled on a plurisubharmonic barrier constructed in \cite{Ber15}).

As a result the envelope is a solution (i.e. a weak geodesic) of the problem (\ref{ss}). The uniqueness of solutions to Monge-Amp\`ere equations follows 
from the comparison principle- see \cite{Gut}.

Another elementary observation is that $u$ satisfies the inequality
$$u(x,t)\leq t\varphi_1(x)+(1-t)\varphi_0(t).$$
Coupling these lower and upper Lipschitz barriers with the convexity of $u$ it easily follows that $u$ is globally Lipschitz in 
$\overline{U\times(0,1)}$ (the complex analogue of this fact was proven in \cite{Abj19}).

Gathering together the conclusions above we obtain the following proposition:
\begin{prop}\label{basicregularity}
Let $U$ be a smoothly bounded strictly convex domain, and $\varphi_0,\varphi_0\in\mathcal S$ . Then the envelope $u$ satisfies the following proprieties:\\
i)  $\det(D^2_{x,t}u)=0$ in $U\times (0,1)$.\\
ii) $u=\phi$ on  $\partial(U\times (0,1))$.\\
iii) $|D u |_{(U\times[0,1])}\leq C .$
\end{prop}

Next lemma is borrowed from \cite{Wan95}. It gives a sufficient condition to glue two convex functions.
\begin{lem}\label{Wang}
 Let $U_1, U_2$ be two domains in $\R^n$ with disjoint interiors. Suppose that the convex functions $u_i$ solve $det(D^2u_i)=0$ in $U_i$ respectively.
 Suppose that $u_1=u_2$ and $Du_1=Du_2$ on $\partial U_1\cap\partial U_2$. Then the function
 $$u(x)\begin{cases} u_1(x), &x\in U_1;\\
 u_2(x), &x\in U_2;\\
 u_1(x)=u_2(x), &x\in \partial U_1\cap\partial U_2
    
   \end{cases}
$$
is convex in the interior of $\overline{U_1\cup U_2}$ and solves $det(D^2u)=0$ there.
\end{lem}

In the proof of the $C^{1,1}$ regularity we shall need the  basic facts from Section 1 in \cite{CNS86}. Recall that a domain
$V\in \mathbb R^m, m>1$ satisfies the {\it truncated cone condition} if the following holds: there exist constants $\varepsilon, \delta>0$ such that
for every $y\in V$ there is a truncated cone
\begin{align*}
K(y):=\lbrace x\neq y| |x-y|<\varepsilon,\ {\rm and\ the\ angle\ between}\\
x-y\ {\rm and\ some\ unit\ vector\ is\ less\ than}\ \delta\rbrace
\end{align*}
which is contained in $V$. 

For our purposes it is sufficient that the domain $U\times(0,1)$ is bounded and convex  and thus satisfies the truncated cone condition.

The following two results are contained in the aforementioned Section 1 in \cite{CNS86}:
\begin{lem}\label{LemmaA}
Let $V$ be a convex domain satisfying the truncated cone condition. Given any convex, uniformly Lipschitz function $v$ satisfying for some constant $C>0$
and all $x,y\in V$ the bound
\begin{equation}\label{global}
|v(x)-v(y)-(x-y).Dv(y)|\leq C|x-y|^2
\end{equation}
one has
$$|Dv(x)-Dv(y)|\leq B|x-y|$$
for some constant $B$ dependent only on $V$ and $C$. In particular such a function is  $C^{1,1}$ regular.
\end{lem}
We sketch the main idea for the sake of completeness.
\begin{proof}
Subtracting a linear function if necessary one may assume that $v(x)=0$ and $Dv(x)=0$. As $v$ is uniformly Lipschitz it
suffices to prove that
$|Dv(y)|\leq B|x-y|$ for all $y$ such that $|x-y|<\varepsilon$ with $\varepsilon$ being the constant in the truncated cone condition.

The truncated cone condition implies that for any vector $\eta\neq 0$ there is a point $z$ in $K(y)$ such that
$|(z-y).\eta|\geq \alpha |z-y||\eta|$ for some uniform $\alpha>0$ dependent only on $\delta$. In particular there is such a $z$
satisfying $|z-y|=|x-y|\leq\varepsilon$ and $|(z-y).Du(y)|\geq \alpha |z-y||Du(y)|$. But it is easily seen that
$$|(z-y).Du(y)|\leq C(|z-x|^2+|z-y|^2+|x-y|^2)\leq 6C|x-y|^2$$
and the proof follows.
\end{proof}
Next lemma from \cite{CNS86} shows that the global inequality (\ref{global}) follows from its localized version:
\begin{lem}\label{LemmaB}
Suppose that $v$ is a differentiable uniformly Lipschitz convex function in a convex domain $V$, such that for every $x_0\in V$ there is
a constant $\epsilon(x_0)>0$ such that for any $x\in V, |x-x_0|\leq \epsilon(x_0)$ one has
\begin{equation}\label{local}
|v(x)-v(x_0)-(x-x_0).Dv(x_0)|\leq C|x-x_0|^2
\end{equation}
for some uniform constant $C>0$.
 Then inequality (\ref{global}) holds with the same constant $C$. In particular $v$ is $C^{1,1}$ regular.
\end{lem}

\begin{proof}
 Fix some constant $E>C$ and suppose that 
$$|v(x)-v(y)-(x-y).Dv(y)|\geq E|x-y|^2$$
 for some pair of points $x$ and $y$ in $V$.
 By assuption for $x$ fixed there is a closest to $x$ such point $z$. On the line segment $[x,z]$ the function
 $f(y):=E|x-y|^2-(v(y)-v(x)-(y-x).Dv(x))$ is positive near $x$ and vanishes at the end points. At an interior maximum $x_0$
 one has 
$$f(x_0+h(z-x))+f(x_0-h(z-x))-2f(x_0)\leq 0$$
 for all $h>0$ small enough
 which means that
 $$2Eh^2|z-h|^2-[v(x_0+h(z-x))+v(x_0-h(z-x))-2v(x_0)]\leq 0.$$
 For $h$ sufficiently small the latter inequality coupled with inequality (\ref{local}) impies that
 $$2(E-C)h^2|z-x|^2\leq 0,$$
 a contradiction. 
\end{proof}

We shall need the following classical fact:
\begin{lem}\label{toric}
Let $\Omega$ be a Reinhardt domain in $\mathbb C^n$ and $u$ be a bounded plurisubharmonic function on $\Omega$, invariant with respect to the toric action. Then:
\begin{enumerate}
 \item The image $U$ of the mapping 
 $$\Omega\ni z\rightarrow (log|z_1|,\cdots,log|z_n|)\in\mathbb R^n$$
 is a domain in $\mathbb R^n$. $\Omega$ is pseudoconvex if and only if $U$ is convex.
 \item The function $v(x):=u(e^{x_1},\cdots,e^{x_n})$ is a convex function in $U$. Reversely for any bounded convex function $v$ on $U$ the function $u$ defined through this formula extends to a toric
 plurisubharmonic function on $\Omega$.
 \item $(dd^cu)^n=0$ if and only if $det(D^2v)=0$- the equivalence continues to hold for bounded singular $u$ and $v$ and then the equalities are understood in weak sense
 of measures- see \cite{Gut,Kol05}.
\end{enumerate} 
\end{lem}

In the proof of Corollary \ref{corol} we shall need some basic geometric facts regarding unbounded convex domains. Recall that if
$\overline{U}$ is the closure of such an unbounded convex domain the {\it characteristic cone} for the point $x\in\overline{U}$

$$\Gamma:=\{r\in\mathbb R^n\ |\ \forall t\in[0,\infty)\ x+tr\in \overline U,\}$$
is nonempty and independent of the base point $x$. 

The following lemma says that very long line segments in $\overline{U}$ with one end point in a fixed compact region must be almost parallel to
a direction from the characteristic cone:
\begin{lem}\label{parallel}
Let $U$ be an unbounded convex domain and let $V$ be a compact subset of $\overline{U}$. Then for every $\varepsilon>0$ there is a positive
constant $h$, dependent on $U, V$ and $\varepsilon$ so that if $x^0\in V, x^1\in \overline{U}$ and the length of $[x^0,x^1]$ is more than $h$ then
there exists a vector $r\in\Gamma$ such that the angle between $r$ and $[x^0,x^1]$ is less than $\varepsilon$.
\end{lem}
\begin{proof}
 Suppose not. Then for every $n\in\mathbb N$ there are points $x^{n,0}\in V, x^{n,1}\in \overline{U}$ at distance at
 least $n$ with the direction of $[x^{n,0},x^{n,1}]$ separated from the $\Gamma$ directions. Let $\bar{x}$ be a cluster point of $x^{n,0}$.
 Then the directions of the line segments $[\bar{x},x^{n,1}]$ for all $n$ large enough and belonging to the sequence defining $\bar{x}$ are also
 separated from the $\Gamma$
 directions which is a contradiction, since the slopes of these segments, after taking another subsequence if necessary,
 converge to a direction from $\Gamma$.
\end{proof}

The last result we shall need is a slight modification of Lemma 3.2 from \cite{LW15}. It says that if for any point $(x,t)\in U\times(0,1)$ the solution to the problem (\ref{ss}) is linear
along a line segment $[(\xi,0),(\eta,1)]$ then the end points depend smoothly on  $x$ and $t$.
\begin{lem}\label{LiWang}
Let $U$ be a convex domain in $\mathbb R^n$ (possibly unbounded). Let $u$ be a convex solution to the problem 
\begin{equation}\label{LWDirichlet}
\begin{cases}
\det(D^2 _{x,t}u)=0\;{\rm in}\ \;\; U\times (0,1);\\
u=\varphi\;\; {\rm in}\ \;\; U\times \{0\};\\
u=\psi\;\; {\rm in}\ \;\; U\times \{1\};\\
u=0\;\; {\rm in}\ \;\; \partial U\times (0,1),
\end{cases}
\end{equation}
with $\varphi,\psi\in \mathcal S$. Suppose that for any $(x,t)\in U\times(0,1)$ there is a unique line segment $L=[(\xi,0),(\eta,1)]$ containing $(x,t)$
and
$\xi, \eta\in U$, so that $u$ is linear along $L$. Then $\xi=\xi(x,t)$ and $\eta=\eta(x,t)$ are smooth functions in 
$\overline{U\times(0,1)}$.
\end{lem}
\begin{proof}
Fix a point $(x,{t})\in U\times(0,1)$. Linearity of $u$ along $L$ forces the
equality
\begin{equation}\label{firstequalitygradients}
  D_x\varphi(\xi({x},{t}))=D_xu(\xi({x},{t}),0)=D_xu(\eta({x},{t}),1)=D_x\psi(\eta({x},{t})).
 \end{equation}
 as $D^2\varphi, D^2\psi$ are smooth strictly positive matrices for all $x\in \overline{U}$ it follows by the implicit function theorem that $\eta$
 is a smooth function of $\xi$ and vice versa.
 
 We have
 \begin{equation}\label{helpful}
 \eta=(D\psi)^{-1}(D\varphi(\xi)),\ D^2\varphi(\xi)= D^2\psi(\eta).D_{\xi}\eta.
 \end{equation}

 Note that this smooth dependence holds up to $\partial U$ thanks to the assumption that $\varphi,\psi$ are strictly convex and smooth up to the boundary.
 
 It remains to check that $\xi$ is a smooth function of $x$ and $t$. To this end we note that (\ref{helpful}) and
 $$(1-t)(\xi,0)+t(\eta,1)=(x,t)$$
 imply
 $$F(t,x,\xi):=(1-t)\xi+tD\psi^{-1}(D\varphi(\xi))-x=0.$$
 It thus suffices to check that $D_\xi F$ is invertible. But following Lemma 3.2 in \cite{LW15} we check that
 $$det[D_\xi F(t,x,\xi)]=det[(1-t)Id+t(D^2\psi)^{-1}((D\varphi(\xi)))\times D^2\varphi(\xi)]$$
 $$=det[(1-t)Id+t(D^2\psi)^{-\frac12}\times D^2\varphi\times (D^2\psi)^{-\frac12}]>0$$
 as the matrix is a convex combination of two strictly positive matrices. Again the assumption $\varphi,\psi\in \mathcal S$
 guarantees that the smooth dependence continues up to the boundary.
\end{proof}

\section{$C^{1,1}$ Regularity}
In this section we provide a proof of Theorem \ref{22}. The proof, except for the last step, copies the classical argument of Caffarelli-Nirenberg-Spruck from \cite{CNS86}. We sketch the reasoning for the sake of completeness. 
\begin{lem}\label{lem1} Let $(x^0,t^0)$ be any point in $U\times(0,1)$. Subtracting a linear function if necessary, we may suppose that 
$$ u\geq 0,\;\;\; u(x^0,t^0)=0.$$
Then $(x^0,t^0)$ is in the convex hull of $(n+1)$ points (not necessarily distinct)\newline $(x^1,t^1), (x^2,t^2),...(x^{n+1},t^{n+1})$ in $\partial (U\times (0,1))$ with $u(x^i,t^i)=0$ for all $i\in \{1,2,...n+1\}.$
\end{lem}
\begin{proof}
This is Lemma 2 from \cite{CNS86}.
By Caratheodory's theorem it suffices to show that $(x^0,t^0)$ is in the convex hull of 
$$ Z=\{(x,t)\in \partial (U\times (0,1))| u(x,t)=0\}.$$

Suppose on contrary that  this is not true. Then  there is a hyperplane $l$ separating $(x^0,t^0)$ from $Z$; i.e. there is an affine function $l$ such that
$$l(x^0,t^0)>0\; and\; l(x,t)< 0\; for\; all\; (x,t)\in Z.$$
Every point $(x,t)$ in $Z$ satisfies after possibly a translation and rotation the following inequality:
$$ x_n<-\epsilon< \epsilon < x_n^0.$$

As $Z$ is a compact subset of $\partial(U\times(0,1))$ there is a positive constant $a$, such that $u(x,t)\geq a$ on $\{x_n\geq 0\}\cap\partial(U\times(0,1))$.
We consider the  function $$v(x,z)= \delta x_n$$
for a fixed positive constant $\delta$.
Fix a point $(x,t)\in\partial(U\times(0,1))$. If $x_n<0$  we have $v\leq u$ at $(x,t)$. If in turn $x_n\geq 0$, we have $u(x,t)\geq a$ and  by choosing $\delta$ small enough we obtain
$$v\leq u \; on \; \partial(U\times (0,1)).$$
Then by the comparison principle we infer that 
$$v(x,t)\leq u(x,t)\; \;\forall(x,t)\in U\times (0,1).$$
This implies that $u(x^0,t^0)\geq \delta\epsilon$, which contradicts the fact that $u(x^0,t^0)=0$.
\end{proof}
Fix now a neighborhood $W$ of $\partial U\times\{0,1\}$, such that 
$$V:=U\times(0,1)\setminus \overline{W}$$
is convex.
With the aid of Lemmas \ref{LemmaA} and \ref{LemmaB}, Theorem \ref{22} follows from the following bound: for every point $(x,s)$ in $V$
there is a positive $\epsilon(x,s)$ dependent on $(x,s)$ and the data in Problem (\ref{ss}), such that for any $(y,t)\in V, 
|(y,t)-(x,s)|<\epsilon(x,s)$ 
one has
\begin{equation}\label{2}
|u(y,t)-u(x,s)-(y-x).D_xu(x,s)-(t-s)D_t(x,t)|
\end{equation}
\begin{equation*}\leq C(|x-y|^2+|t-s|^2)
\end{equation*}
for a constant $C$ depending on $V$ and the $C^2$ norm of $\varphi_0$,$\varphi_1$ (but independent on $(x,s)$).

 \begin{proof}[Proof of inequality (\ref{2})]
 We fix $(\bar{x},\bar{t})\in V$. After possibly an addition of an affine function $l_{\bar{x},\bar{t}}(x,t)$ we can suppose 
 $$ u \geq 0,\;\;  u(\bar{x},\bar{t})=0 \; \text{and}\; D u(\bar{x},\bar{t})=0.$$
 The problem (\ref{ss}) for this new function $u$ becomes:
 \begin{equation}\label{sss}
\begin{cases}
\det(D^2 _{x,t}u)=0\;in \;\; U\times (0,1),\\
u=\varphi_0+l_{\bar{x},\bar{t}}\;\; on \;\; U\times \{0\},\\
u=\varphi_1+l_{\bar{x},\bar{t}}\;\; on \;\; U\times \{1\},\\
u=l_{\bar{x},\bar{t}}\;\; on \;\; \partial U\times (0,1).
\end{cases}
\end{equation}
The inequality we need to prove reads
\begin{equation}\label{am}
u(x,t) \leq C(|x-\bar{x}|^2+|t-\bar{t}|^2)
\end{equation}
for all $(x,t)$ sufficiently close to $(\bar{x},\bar{t})$.

Up to now we were simply following the argument from \cite{CNS86}. To proceed we need the following fact:

{\bf Claim}: $Z$ contains at most one point with $t$-coordinate equal to $0$, no points with $t\in(0,1)$ and at most one point with $t$-coordinate equal to $1$. In particlar the convex hull of $Z$ is a line segment.

In order to show the claim we divide the boundary of $U\times(0,1)$ 
$$\partial(U\times (0,1))$$
$$=\partial U \times (0,1)\cup{U}\times \{0,1\}\cup(\partial U\times\lbrace0\rbrace\cup\partial U\times\lbrace1\rbrace)$$
$$=I_1\cup I_2\cup I_3$$
into  three types of boundary points.

Strict convexity of $\varphi_0$ and $\varphi_1$ shows that no two points from $Z$ could belong to $t=1$ hyperplane, as well as to the $t=0$ hyperplane.
It is easy then  to see that neither of the points from $Z$ could belong to $I_1$. Indeed if this were the case for $(x,t)$, say, then the affine function $l_{\bar{x},\bar{t}}$ is nonnegative
on the segment $(x,s),\ s\in[0,1]$ and vanishes on $(x,t)$, hence it vanishes at $(x,0)$ and $(x,1)$. As $Z$ contains other points on $\partial(U\times(0,1))$ we get a contradiction with our previous observation.

The case when two points belong to $I_3$ is also easily ruled out, as once again $l_{\bar{x},\bar{t}}$ vanishes on the plane spanned by $(x^i,0), (x^i,1),\ i=0,1$.

Thus we can assume without loss of generality that $(x^1,t^1)\in I_2$ and that $t^1=1$. In case $(x^0,t^0)$ also belongs to $I_2$ we assume, switching the role of the end points if necessary, that $(\bar{x},\bar{t})$ is closer to $(x^1,t^1)$ than to $(x^0,t^0)$.

Consider a ball in $U\times (0,1)$ with center $(\bar{x},\bar{t})$ and radius $\epsilon$ with $\epsilon$ to be determined. For any $(x,t)$ in the $\epsilon$- ball let $(\hat{x},\hat{t})$ be the point  where the ray from $(x^0,t^0)$ to $(x,t)$ strikes the boundary of $U\times(0,1)$
for the second time. As $(x^1,t^1)\in I_2$ by choosing the $\epsilon$ sufficiently small we will have $(\hat{x},\hat{t})\in U\times \{1\}$. If
$(x,t)=s(\hat{x},1)+(1-s)(x^1,0)$ by convexity of $u$ and $u(x^0,0)=0$, we have 
$$u(x,t) \leq s\tilde{\varphi}_1(\hat{x})+(1-s)u(x^0,0)\leq \tilde{\varphi}_1(\hat{x}),$$
 where $\tilde{\varphi}_1=\varphi_1+l_{\bar{x},\bar{t}}$. By Taylor expansion (note that $D_x\tilde{\varphi}_1(x^1)=0$ as it is a local minimum point on the $t=1$ hyperplane) we have
$$ \tilde{\varphi}_1(\hat{x})-\tilde{\varphi}_1(x^1)\leq C|\hat{x}-x^1|^2$$
for a constant $C$ dependent on the $C^2$ norm of $\varphi_1$ and the Lipschitz norm of $u$.

If we could now  prove that 
\begin{equation}\label{needed}
|\hat{x}-x^1|^2\leq C(|x-\bar{x}|^2+|t-\bar{t}|^2),
\end{equation}
for some $C$ under control we are done. Indeed, we have 
$$\frac{|\hat{x}-x^1|}{|(x^0,t^0)-(x^1,t^1)|}=\frac{\sin(\beta)}{\sin(\alpha+\beta)}
$$
and 
$$\frac{|(\bar{x},\bar{t})-(x^0,t^0)|}{|(x,t)-(\bar{x},\bar{t})|}=\frac{\sin(\theta)}{\sin(\beta)},$$
where $\alpha$  is the angle between the line $((x^1,1),(x^0,0))$ and the hyperplane $t=1$, $\beta$ is the angle between $((x^0,0),(x^1,1))$ and $((x^0,0),(\hat{x},1))$
and $\theta$ is the angle between $((\bar{x},\bar{t}),(x,t))$ and $((x^0,0),(\hat{x},1))$. From these two equations we obtain 
$$\frac{|\hat{x}-x^1|}{|(x,t)-(\bar{x},\bar{t})|}=\frac{\sin(\theta)}{\sin(\alpha+\beta)}\frac{|(x^0,t^0)-(x^1,t^1)|}{|(\bar{x},\bar{t})-(x^0,t^0)|}.$$

Note that the angle $\alpha+\beta$ is uniformly bounded from below by a constant dependent on $diam(U)$. Also, trivially, $sin(\theta)\leq 1$.

In order to bound the ratio $\frac{|(x^0,t^0)-(x^1,t^1)|}{|(\bar{x},\bar{t})-(x^0,t^0)|}$ consider two cases:

Case 1. If $(x^0,t^0)\in I_3$, then, as $(\bar{x},\bar{t})\in V$ the quantity $|(\bar{x},\bar{t^0})-(x^0,t^0)|$ is bounded from below by a constant dependent on $W$, while $|(x^0,t^0)-(x^1,t^1)|$ is bounded from above by a constant dependent only on $diam(U)$.

Case 2. If $(x^0,t^0)$ also belongs to $I_2$ then, recalling that $(\bar{x},\bar{t})$ is closer to $(x^1,t^1)$ than to $(x^0,t^0)$, we have 
$\frac{|(x^0,t^0)-(x^1,t^1)|}{|(\bar{x},\bar{t})-(x^0,t^0)|}\leq 2$.

In both cases we get that (\ref{needed}) holds, and hence the proof is concluded.
 \end{proof}
\begin{rem} In the proof we haven't made use of the smoothness of $\partial U$. In particular the regularity still holds for any strictly convex bounded domain with $C^{1,1}$ boundary.
\end{rem}

The proof of Corollary \ref{corol} follows similar lines once we translate the problem on the logarithmic image of $\Omega$. Below we sketch the details.
\begin{proof}[Proof of Corollary \ref{corol}] 
Let $U$ be the logarithmic image of the Reinhardt domain $\Omega$. Then in $U\times(0,1)$ the function 
$$v(x,t):=\phi(e^{x_1},\cdots,e^{x_n},e^{t})$$
solves the homogeneous real Monge-Amp\`ere equation
with the corresponding boundary values.

If the axis set $M$ is empty then $U$ is bounded, smooth, strictly convex domain and Theorem \ref{22} applies yielding almost global $C^{1,1}$ smoothness of $v$ which in turn implies the claimed regularity of $\phi$.

If $M\neq\emptyset$ the domain $U$ is unbounded. Again we assume, adding an affine function if necessary, that $v\geq 0,\ v(\bar{x},\bar{t})=0$ 
for some fixed point $(\bar{x},\bar{t})\in U\times(0,1)$. It is easy to see that $\{(x,t)\in \overline{U\times(0,1)}|\ v(x,t)=0\}$ is convex and has no extremal points in $U\times(0,1)$. 
Indeed, arguing in a small ball around such a point we can repeat the argument from the bounded case.

Consider 
$$Z:=\{(x,t)\in\partial(U\times(0,1))\ |\ v(x,t)=0\}.$$

The argument above shows that $(\bar{x},\bar{t})$ is in the convex hull of $Z$ unless there is an infinite ray passing though it. 
The end  point of such a ray has to be in $\partial U\times(0,1)$ and it has to be parallel to the $t=const$ hyperplanes. But the initial function $v$ is bounded and linear on such a ray, hence it is constant there. The position of the end-point forces that $v$=0 along the ray, which contradicts the negativity of the initial $v$ in $U\times(0,1)$. 
As a result $(\bar{x},\bar{t})$ has to be in the convex hull of $Z$.

Just as in the previous proof there is at most one point in $Z$ with $t$-coordinate equal to $0$ and at most one 
with $t$-coordinate equal to $1$, for otherwise we get a contradiction with the strict convexity of the boundary values. 
Again this precludes the existence of a point $(\tilde{x},\tilde{t})$ from $Z$ on $\partial U\times(0,1)$ for it would imply  that $(\tilde{x},t)\in Z$ for every $t\in[0,1]$. 
As a result $Z$ consists of two points and the line segment segment joining them passes through $(\bar{x},\bar{t})$. Let the end points be $(x^0,0), (x^1,1)$ with $x^0\in\overline{U},\ x^1\in \overline{U}$ with at most one being in $\partial U$.

What remains to be done is to show that this line segment is uniformly bounded in length, or equivalently it is not too parallel to the
$t=const$ hyperplanes.

To this end fix two Reinhardt pseudoconvex neighborhoods $\Theta\supset \Theta'$ of $M$ in $\Omega$ which yields  convex neighborhoods of infinity $W, W'$ in $U$. We can assume that $V:=U\setminus W$ is bounded. Shinking  $W'$ if necessary we can assume that the distance between $W'$ and $V$ is equal to  $diam{V}$

Given a line segment $[(x^0,0),(x^1,1)]$ as above with $(\tilde{x},\tilde{t})$ in $V$ we  suppose, without loss of generality, that $(x^0,0)\in V$. If $(x^1,1)$ is in
$V'$ the length of the segment is bounded by a constant dependent on the diameter of $V$  and we conclude as in the bounded domain case.

Hence from now on we assume that $(x^1,1)\in \overline{U}\setminus V'$. 

{\bf Case 1}. Assume that $x^1\in U$.

By Lemma \ref{parallel} we can assume that the vector $x^1-x^0$ almost belongs to the characteristic cone $\Gamma$ of $U$: there is a unit vector $r\in \Gamma$, such 
that the angle between $r$ and $\kappa:=\frac{x^1-x^0}{|x^1-x^0|}$ is less than any preassigned $\varepsilon>0$ if the line segment is long enough.
But then $\kappa=r+\theta$, where $\theta\in\mathbb R^n$ is a vector of length at most $\sqrt{2-2cos(\varepsilon)}$.

{\bf Claim}: There is a $\delta>0$ dependent only on $U,V$ and $\varphi$ such that
$$-D_r\varphi(x^0)\geq \delta.$$

Note that as $s\rightarrow \tilde{\varphi}(s)=\varphi(x^0-sr)$ is a convex function on an infinite ray it follows that
$\tilde{\varphi}$ is increasing in $s$ and $\tilde{\varphi}'(s)=-D_r\varphi(x^0-sr)\geq 0$. Obviously 
$$lim_{s\rightarrow -\infty}-D_r\varphi(x^0-sr)=0.$$
Hence
$$-D_r\varphi(x^0)=\int_{-\infty}^{0}D^2\varphi(x^0-rs)(r,r)ds.$$

Note that the inegral is taken over a ray that intersects $V'$ in a line segment of length at least $diam(V)$ and that $D^2\varphi$ is uniformly
convex there. Thus the claim follows.

Recall that 
\begin{equation}\label{xyzt}
-D_r\psi(x^1)\geq -D_r\varphi (x^0)
\end{equation}
with equality unless $x^0\in \partial U$. This implies 
$$\delta<-D_r\psi(x^1)=-D_{\kappa}\psi(x^1)+D_{\theta}\psi(x^1)\leq -D_{\kappa}\psi(x^1)+C\varepsilon.$$

Fixing $\varepsilon$ sufficiently small we obtain, assuming that $|x^1-x^0|$ is large enough,
$$-D_{\kappa}\psi(x^1)\geq\frac\delta2.$$

But on the other hand
$$-D_{\kappa}\psi((1-s)x^1+sx^0)=-D_{\kappa}\psi(x^1)$$
$$+|x^1-x^0|\int_0^{s}D^2\psi
((1-t)x^1+tx^0)(\kappa,\kappa)dt>-D_{\kappa}\psi(x^1)$$
for any $s\in(0,1)$. Thus
$$\psi(x^0)-\psi(x^1)=\int_0^{1}-D_{(x^1-x^0)}\psi((1-s)x^1+sx^0)ds$$
$$\geq |x^1-x^0|\frac{\delta}2,$$
a contradiction with the uniform bound of $\psi$ if $ |x^1-x^0|$ is too large.

{\bf Case 2}. Let now $x^1$ belong to $\partial U$ forcing $x^0\in U$. The problem with the previous argument is that the inequality (\ref{xyzt}) would be now in the wrong direction.

Instead, we observe that  
$$0=\psi(x^1)-\varphi(x^0)-D_tv(\bar{x},\bar{t})-Dv(\bar{x},\bar{t}).(x^0-x^1),$$
which yields
$$|D_{\kappa}\varphi(x^0)||x^1-x^0|=|Dv(\bar{x},\bar{t}).(x^1-x^0)|\leq C$$
as $v$ is locally uniformly Lipschitz and $\psi$, $\varphi$ are bounded.

Thus 
\begin{equation}\label{lasttt}
|D_{\kappa}\varphi(x^0)|\leq \frac{C}{|x^1-x^0|},
\end{equation}
 but on the other hand from the proof in the first case we know that  $-D_r\varphi(x^0)\geq \delta$ is uniformly positive for any $r\in\Gamma$. Once again exploiting the fact that there is an $r$ in $\Gamma$ very close to $\kappa$ if $|x^1-x^0|$ is sufficiently large, we get $-D_{\kappa}\varphi(x^0)\geq\frac\delta2$, a contradiction with (\ref{lasttt}) if $[x^0,x^1]$ is too long.

As a result we obtain a bound on $|x^1-x^0|$ and hence on the length of the line segment and the proof concludes as in the bounded domain case.

\end{proof}
\begin{rem} One can use also an indirect argument to prove that line segments must have bounded length. Indeed, suppose that
there are sequences of points $(x_n^0,0),\ (x_n^1,1)$, such that $x_n^0$, say, varies in a relatively compact subset of $\overline{U}$, and $v$ is linear along
$[(x_n^0,0),(x_n^1,1)]$. Taking a subsequence if necessary, these segments will converge to a ray in $U\times\{0\}$. As $v$ is bounded, this convergence implies
the constancy of $v(x,0)=\varphi(x)$ along the ray, a contradiction with the strict convexity of $\varphi$. The direct argument however provides an explicit
bound and hence one has a better control on the $C^{1,1}$ bound of the solution.
\end{rem}

 \begin{rem} It is likely that the weak geodesic is also $C^{1,1}$ smooth near $M\times A$.
\end{rem}
 \section{real example}
 In this section we discuss an example of a solution to the problem (\ref{ss}) with smooth data, which fails to be $C^{1,1}$ up to  the boundary.
\begin{exa}\label{realexample} We take \;$\Omega =(-1,1)$, $\varphi_0(x)=2(x^2-1)$, $\varphi_1(x)=x^2-1$. We consider  the following function
\begin{equation*}u(x,t):=
\begin{cases}
2(1-t)\left((\frac{x+t}{1-t})^2-1\right)& \text{if}\;\frac{x+t}{1-t}<\frac{-1}{2};\\
\frac{2x^2}{1+t}+t-2 & \text{if}\; \frac{x+t}{1-t}\geq\frac{-1}{2}\; \text{and}\; \frac{x-t}{1-t}\leq \frac{1}{2};\\
2(1-t)\left((\frac{x-t}{1-t})^2-1\right)& \text{if} \;\frac{x-t}{1-t}\geq\frac{1}{2}.
\end{cases}
\end{equation*}
\end{exa}
It is obvious $u(x,0)=\varphi_0(x)$ and $u(x,1)=\varphi_1(x)$.
We are going to check that the function $u$ is convex and it is indeed a solution of the above Dirichlet problem.
To this end we compute the Hessian matrix of $u$ in the three cases. 

Case 1:  $u_1(x,t): =2(1-t)\left((\frac{x+t}{1-t})^2-1\right)$, when $\frac{x+t}{1-t}<\frac{-1}{2}$. By straightforward computation we obtain
\begin{equation*} D^2u_1(x,t)=
\begin{pmatrix}
   \frac{\partial^2 u}{\partial^2x }       & \frac{\partial^2 u}{\partial x\partial t }  \\
     \frac{\partial^2 u}{\partial x\partial t }     & \frac{\partial^2 u}{\partial^2t }
\end{pmatrix}=\begin{pmatrix}
   \frac{4}{1-t}     & \frac{4(1+x)}{(1-t)^2}  \\
      \frac{4(1+x)}{(1-t)^2}   & \frac{4(1+x)^2}{(1-t)^3 }
\end{pmatrix}.
\end{equation*}
Obviously $D^2u_1$ is semi-postive with vanishing determinant, which implies that $u$ is convex with respect to $(x,t)$. Note also that the second order derivatives blow-up as $(x,t)\rightarrow (-1,1)$.

Case 2:  $u_2(x,t):=\frac{2x^2}{1+t}+t-2$, when $\frac{x+t}{1-t}\geq\frac{-1}{2}\; \text{and}\; \frac{x-t}{1-t}\leq \frac{1}{2}$. In this case the Hessian reads 
\begin{equation*} D^2u_2(x,t)=\begin{pmatrix}
   \frac{4}{1+t}     & \frac{-4x}{(1+t)^2}  \\
      \frac{-4x}{(1+t)^2}   & \frac{4 x^2}{(1+t)^3 }
\end{pmatrix}.
\end{equation*}
 Again $D^2u_2$ is semi-postive with vanishing determinant. Note however that the second derivatives are bounded
 which is in line with Theorem \ref{22}.

Case 3:  $u_3(x,t): =2(1-t)\left((\frac{x-t}{1-t})^2-1\right)$, when $\frac{x-t}{1-t}\geq\frac{1}{2}$. In this case

\begin{equation*} D^2u_3(x,t)=\begin{pmatrix}
   \frac{4}{1-t}     & \frac{4(x-1)}{(1-t)^2}  \\
      \frac{4(x-1)}{(1-t)^2}   & \frac{4(x-1)^2}{(1-t)^3 }
\end{pmatrix}.
\end{equation*}
Once again the Hessian is semi-positive with vanishing determinant. In this case the second order derivatives blow-up as $(x,t)\rightarrow (1,1)$.

We shall use Lemma \ref{Wang} to show that $u_i,\ i=1,2,3$ glue together to a convex function.  We  denote by $U_1$, $U_2$ and $U_3$ the domains of definition of $u_1$, $u_2$  and $ u_3$ respectively.
On $\partial U_1\cap \partial U_2=\{\frac{x+t}{1-t}=\frac{-1}{2}\},$
we have 
$$u_{1}(x,t)=\frac{3(t-1)}{2}=u_2(x,t),$$
while
$$ D u_1(x,t)=(-2,\frac{1}{2})=D u_2(x,t). $$

We have also along $\partial U_2\cap \partial U_3=\{\frac{x-t}{1-t}=\frac{1}{2}\}$
$$u_3(x,t)=\frac{3(t-1)}{2}=u_2(x,y),$$
while
\begin{equation*}
D u_3(x,t)=(2,\frac{1}{2})=D u_2(x,y).
\end{equation*}

Thus $u_1$, $u_2$ and $u_3$ glue together to a convex function $u$ matching the given boundary data. By Lemma \ref{Wang} $det(D^2u)=0$ in the weak sense.

\section{Complex setting}
In this section we construct an example of a geodesic which  joins two smooth strictly  plurisubharmonic functions but is not $C^{1,1}$ up to the boundary of $\Omega\times A.$
\begin{exa}\label{complexexample} Let $\Omega$ be the unit disc in $\C$.We join $\varphi_0(w)=2(|w|^2-1)$ and $\varphi_1(w)=|w|^2-1$ by 
\begin{equation*}
u(w,z)=
\begin{cases} 
2\frac{|w|^2}{|z|^{\log2}}+\log|z|-2 \;\; if\;\;\frac{|w|}{|z|^{log\sqrt{2}}}<\frac{1}{\sqrt{2}};\\
2(1-\log|z|) (e^{\frac{\log|w|^2}{1-\log|z|}}-1)\;\; {\rm otherwise}. \;\; 
\end{cases}
\end{equation*}
\end{exa}
Below we check that $u$ is  plurisubharmonic  with respect to $(w,z)$ and  satisfies $$(dd^cu)^2=0,$$ which means $u$ is a weak geodesic between $\varphi_0$ and $\varphi_1$.

 We check first  the plurisubharmonicity of the two parts of $u$. The complex Hessian of $u$ in $\lbrace\frac{|w|}{|z|^{log\sqrt{2}}}<\frac{1}{\sqrt{2}}\rbrace\cap \mathbb D\times A$ reads

 \begin{equation*}
\begin{pmatrix}
\frac{(\log 2)^2|w|^2}{2|z|^{\log(2)+2}}  &  \frac{-z\bar{w}\log(2)}{|z|^{\log2+2}}   \\
        \frac{-w\bar{z}\log2}{|z|^{\log2+2}}   &  \frac{2}{|z|^{\log2}} 
\end{pmatrix}.
\end{equation*}

The complex Hessian is Hermitian semi-postive with vanishing determinant, which implies that $u$ is plurisubharmonic and satisfies 
$$(dd^cu)^2=0.$$
It can be also checked by direct computation that
all the second order partial derivatives of $u$ remain bounded in this case.

In the remaining part of $\mathbb D\times A$ the complex Hessian of  $u$ is given by.

\begin{equation*}
\begin{pmatrix}
   \frac{(\log|w|^2)^2}{2(1-\log|z|)^3|z|^2}.A(z,w)   &   \frac{\log|w|^2 z\bar{w}}{(1-\log|z|)^2|z|^2|w|^2} .A(z,w) \\
      \frac{\log|w|^2 w\bar{z}}{(1-\log|z|)^2|z|^2|w|^2}.A(z,w)     &  \frac{2}{|w|^2(1-\log|z|)}.A(z,w)
\end{pmatrix},
\end{equation*}
where $A(z,w):=e^{\frac{\log|w|^2}{1-\log|z|}}.$
 
 Again the complex Hessian is Hermitian semi-postive, and hence $u$ is plurisubharmonic with vanishing complex Monge-Amp\`ere 
 operator in this second case. Note that if $(w,z)\rightarrow (w_0,z_0)\in \partial\mathbb D\times\partial A$ the second-order derivatives blow-up.

It remains to check that both parts of $u$ glue together in a plurisubharmonic fashion. Indeed,  to prove that $u$ is plurisubharmonic we 
need to show that $u$ is well glued in a neighborhood of the hypersurface
$$\frac{|w|}{|z|^{log\sqrt{2}}}=\frac{1}{\sqrt{2}}.$$
 We put  $t:= \log|z|$ and $x:=\log|w| $. Then $v(x,t):=u(e^x,e^t)$ becomes
\begin{equation*}
v(x,t)=
\begin{cases} 
v_1(x,t)=2\frac{e^{2x}}{e^{log(2)t}}+t-2 \;\;{\rm if}\;\;x<\log(\sqrt{2})t-\log\sqrt{2}; \\
v_2(x,t)=2(1-t) (e^{\frac{2x}{1-t}}-1)\;\; {\rm if}\;\;x\geq\log(\sqrt{2})t-\log\sqrt{2}.
\end{cases}
\end{equation*}
By Lemma \ref{toric} $v_i,\ i=1,2$ are convex in their domains of definition and $det(D^2v_i)$ vanish.  We will next check that $v$ is globally convex
with vanishing real Monge-Amp\`ere. To this end we shall use Lemma \ref{Wang}. On the line $x=\log(\sqrt{2})t-\log\sqrt{2}$ we have:
$$v_1(x,t)=t-1=v_2(x,t),$$
$$D v_1(x,t)=(-\log(2)+1,2)=D v_2(x,t).$$

 Thus the conditions in Lemma \ref{Wang} are satisfied and  we conclude that $v$ is convex with globally vanishing real Monge-Amp\`ere measure.
 One more application of Lemma \ref{toric} yields that $u$ is plurisubharmonic and $(dd^cu)^2=0$ in the sense of measures, as claimed.
  
  \section{Existence of smooth geodesics}
  
 In this section we shall provide a proof of Theorem \ref{smoothgeodesic}. From the proof of Theorem \ref{22} we know that
 through every point $(\bar{x},\bar{t})$ there is a unique line segment $L_{\bar{x},\bar{t}}$ along which the weak geodesic
 $u$ is linear.  
 
 At this moment we remark that two such segments can meet at the {\it boundary} of $U\times(0,1)$, as Example \ref{realexample} shows.
 
 Denote by
 $\xi=\xi(\bar{x},\bar{t}),\ \eta=\eta(\bar{x},\bar{t})\in \overline{U}$ the points so that $(\xi(\bar{x},\bar{t}),0)$ and $(\eta(\bar{x},\bar{t}),1)$ are the end points of
 $L_{\bar{x},\bar{t}}$. We have already observed in the proof of Theorem \ref{22} that two cases are possible:
 
 {\it Case 1.} Both $\xi$ and $\eta$ belong to $U$.
 
 {\it Case 2.} One of the points $\xi,\eta$ belongs to $\partial U$, while the other belongs to $U$.
 
 In the first case, we recall that 
 \begin{equation}\label{equalitygradients}
  D_x\varphi(\xi(\bar{x},\bar{t}))=D_xu(\xi(\bar{x},\bar{t}),0)=D_xu(\eta(\bar{x},\bar{t}),1)=D_x\psi(\eta(\bar{x},\bar{t})).
 \end{equation}
 \begin{proof}[Proof of Theorem \ref{smoothgeodesic}]
 
 Suppose first that $\partial\varphi(U)\neq\partial\psi(U)$. Observe that both sets are open. Swapping $\varphi$ and $\psi$ if necessary, we may thus suppose that there
 is a point $\xi'\in U$  such that $p'=D\varphi(\xi')$ is a vector that does not belong to $\partial\psi(U)$. As $u$ is $C^{1,1}$ 
 there is a point $(x,t)\in U\times(0,1)$
 sufficiently close to $(\xi',0)$ such that $p:=D_xu(x,t)$ does not belong to $\partial\psi(U)$. Then if $L=[(\xi,0),(\eta,1)]$
 is a segment through $(x,t)$ along which $u$ is linear we must have $\eta\in\partial U$ as otherwise equation \ref{equalitygradients} would be violated.
 
 In a small ball $B\in U$ around $x$ the set  $\{D_xu(z,t)\ |\ (z,t)\in B\}$ is disjoint from $\partial\psi(U)$. This yields a continuous mapping

$$B\ni x\longmapsto \eta(x,t)\in \partial U.$$
Obviously the map cannot be injective and hence we get two points $a_1,a_2\in B$ with the same $\eta$. Continuing the rays 
$[(\eta,1),(a_j,t)), j=1,2$ until they hit the boundary of $U\times(0,1)$ for the second time we obtain two points $(b_1,0),(b_2,0)\in U\times\{0\}$ with the same $\eta$. 
Suppose, translating and rotating the coordiantes if necessary, that $b_1=0$ while $b_2=(b,0,\cdots,0)$ for some $b>0$.
But then $D_{x_1}u$ is constant
on the segments $[(b_j,0),(\eta,1)), j=1,2$. As a result for any $t\in[0,1)$
$$\int_0^1D^2_{x_1x_1}\varphi(sb,0\cdots,0)bds=D_{x_1}\varphi(b,0,\cdots,0)-D_{x_1}\varphi(0)$$
$$=D_{x_1}u(t\eta+(1-t)(b,0,\cdots,0),t)-D_{x_1}u(t\eta,t)$$
$$=\int_0^{1-t}D^2_{x_1x_1}u(t\eta+s(b,0,\cdots,0),t)bds.$$
Note that both integrals integrate positive functions as $u$ is convex. The domain of integration shrinks to a point as $t\rightarrow 1^-$
which implies that $D^2_{x_1x_1}u$ becomes arbitrarily large somewhere along the integral path.

 Suppose now that equality 
 \begin{equation}\label{gradeq}
  \partial\varphi(U)=\partial\psi(U)
 \end{equation}
 holds. Let $(\bar{x},\bar{t})$ be a point in $U\times(0,1)$.
 Then 
 $$v(x,t):=u(x,t)-u(\bar{x},\bar{t})-Du(\bar{x},\bar{t}).(x-\bar{x},t-\bar{t})$$
 is nonnegative and vanishes on the line segment $L_{\bar{x},\bar{t}}$. Suppose, without loss of generality, that the end point
 $(x^0,t^0)$ belongs to $U\times\lbrace0\rbrace$. Then repeating the same argument as in the proof of equation 
 (\ref{equalitygradients}) $D_xv(x^0,0)=0$, thus $D_xu(x^0,0)=D\varphi(x^0)=D_xu(x,t)$, implying that
 $D_xu(x,t)$ belongs to $\partial\varphi(U)=\partial\psi(U)$. We claim that this forces that the second end point $(x^1,t^1)$ belongs to
 $U\times\lbrace1\rbrace$.
 
 Indeed, from the strict convexity of $\psi$ and the equality (\ref{gradeq}) we obtain a unique point $\hat{x}\in U$ so that
 $D\psi(\hat{x})=p.$ But if $x^1\in\partial U$ then $D\psi(x^1).(x^1-\hat{x})\leq p.(x^1-\hat{x})$, as $v$ is nonnegative on
 the segment $[(\hat{x},1),(x^1,1)]$ and vanish at $(x^1,1)$. 
 Then
 $$0<\int_0^1D^2_{x^1-\hat{x},x^1-\hat{x}}\psi(sx^1+(1-s)\hat{x})ds$$
 $$=D\psi(x^1).(x^1-\hat{x})-D\psi(\hat{x}).(x^1-\hat{x})$$
 $$\leq p.(x^1-\hat{x})-p.(x^1-\hat{x})\leq 0,$$
 a contradiction.
 
 As a result for any point $(x,t)$ there are uniquely defined 
 $$\xi(x,t),\eta(x,t)\in U,$$
 so that $u$ restricted to the segment
 $[(\xi(x,t),0),(\eta(x,t),1)]\ni (x,t)$ is linear. Hence
 \begin{equation}\label{linearform}
 u(x,t)=u(t(\eta,1)+(1-t)(\xi,0))=t\psi(\eta(x,t))+(1-t)\varphi(\xi(x,t)).
 \end{equation}
Smoothness of $u$ now follows from Lemma \ref{LiWang}.
 
 It remains to prove strict convexity in the space variables. Fix any vector $T\in\mathbb R^n,\ T\neq 0$. Then
 \begin{equation}\label{q}
  D^2_{TT}u(x,t)=tD^2\psi(\eta)(D_T\eta)^2+(1-t)D^2\varphi(\xi)(D_T\xi)^2
 \end{equation}
 $$+tD\psi(\eta)D^2_{TT}\eta+(1-t)D\varphi(\xi)D^2_{TT}(\xi).$$
 
 Recall that $(1-t)\xi+t\eta=x$ and $D\varphi(\xi)=D\psi(\eta)$. These imply that the last row in (\ref{q}) vanishes.
 On the other hand the strict convexity of $\psi$ and $\eta$ coupled with $(1-t)\xi_T+t\eta_T=T$ yields that the sum $tD^2\psi(\eta)(D_T\eta)^2+(1-t)D^2\varphi(\xi)(D_T\xi)^2$
 is strictly positive. Gathering these, we obtain
 $$D^2_{TT}u(x,t)>0$$
up to $\partial U$, as claimed.
 \end{proof}

\end{document}